\documentclass[11pt]{amsart}

\usepackage{fullpage}
\usepackage[latin1]{inputenc}
\usepackage{amssymb,dsfont}
\usepackage{amsmath}
\usepackage{caption}
\usepackage{subcaption}
\usepackage{tikz}
\usetikzlibrary{arrows}
\usetikzlibrary{decorations.markings}
\usepackage{rotating}
\usepackage{enumerate}
\usepackage{stmaryrd}
\usepackage{array}

\DeclareMathOperator{\sgn}{sgn}

\makeatletter
\providecommand*{\cupdot}{%
  \mathbin{%
    \mathpalette\@cupdot{}%
  }%
}
\newcommand*{\@cupdot}[2]{%
  \ooalign{%
    $\m@th#1\cup$\cr
    \hidewidth$\m@th#1\cdot$\hidewidth
  }%
}
\makeatother

\makeatletter
\newcommand{\xmapsto}[2][]{\ext@arrow 0599{\Mapstofill@}{#1}{#2}}
\def\Mapstofill@{\arrowfill@{\Mapstochar\relbar}\relbar\rightarrow}
\makeatother

\tikzset{*->-*/.style={decoration={
  markings,
  mark=at position .5 with {\arrow{>}},
  mark=at position 0 with { \draw [fill] (0,0) circle [radius=2pt];},
  mark=at position 1 with { \draw [fill] (0,0) circle [radius=2pt];}},postaction={decorate}}}
  
\tikzset{dot/.style={draw,shape=circle,fill=black,scale=.5}}

\tikzset{Bullet/.style={draw,shape=circle,fill=black,scale=.5}}

  \tikzset{*--*/.style={decoration={
  markings,
  mark=at position 0 with { \draw [fill] (0,0) circle [radius=2pt];},
  mark=at position 1 with { \draw [fill] (0,0) circle [radius=2pt];}},postaction={decorate}}}

\newtheorem{thm}{Theorem}[section]
\newtheorem{prop}[thm]{Proposition}
\newtheorem{lem}[thm]{Lemma}
 
\newtheorem{conj}[thm]{Conjecture} 
\newtheorem{definition}[thm]{Definition}
\newtheorem{example}[thm]{Example}

\newcommand{\RR}{\mathbb{R}}  
\newcommand{\ZZ}{\mathbb{Z}}
\newcommand{\ubar}{\underline}
\newcommand{\PP}{\mathcal{P}}

\newcommand{\xx}{\mathbf{x}}

\newcommand{\PPP}{\mathbb{P}}

\newcommand{\EE}{\mathcal{E}}

\newcommand{\omitt}[1]{}

\begin{document}

\title{The frequency of pattern occurrence in random walks}

\author{Sergi Elizalde}
\author{Megan Martinez}
\address{Department of Mathematics, Dartmouth College, 
Hanover, NH 03755}
\email{sergi.elizalde@dartmouth.edu}
\email{megan.a.martinez.gr@dartmouth.edu}

\begin{abstract}
In the past decade, the use of ordinal patterns in the analysis of time series and
dynamical systems has become an important and rich tool. Ordinal patterns (otherwise known
as a permutation patterns) are found in time series by taking $n$ data points at evenly-spaced
time intervals and mapping them to a length-$n$ permutation determined by relative ordering.
The frequency with which certain patterns occur is a useful statistic for such series;
however, the behavior of the frequency of pattern occurrence is unstudied for most models.
We look at the frequency of pattern occurrence in random walks in discrete time and, applying combinatorial methods,
we characterize those patterns that have equal frequency, regardless of probability distribution.
\end{abstract}


 \maketitle
 
 \section{Introduction}

Time series analysis deals with the extraction of information from sequences of data points, typically measured at uniform time intervals. Understanding the characteristics of the data enables better predictions of the future behavior of a phenomenon.
There are a number of different statistical methods that can be applied to the study of time series. A relatively new method involves
the analysis of its ``ordinal patterns.'' This approach, pioneered in the dynamical systems community by Bandt, Keller and Pompe~\cite{BP,BKP} and surveyed by Amig\'o~\cite{Amigobook}, is particularly amenable to a combinatorial treatment. For one-dimensional deterministic time series that arise from iterating a map, combinatorial analyses of the ordinal patterns for specific maps have appeared in \cite{AEK,Elishifts,Elibeta,ArcEli}. However, perhaps surprisingly, very little is known about the behavior of ordinal patterns in a random setting. 

In this paper we study, from a combinatorial perspective, the ordinal patterns that occur in random walks. We provide a combinatorial characterization of equivalence classes of patterns that occur with the same probability in any random walk.
We expect such a characterization to be useful in many of the applications of ordinal patterns in random walks that have recently appeared in the dynamical systems literature.
Such applications include the analysis of stock indices and economic indicators, both to quantify the randomness of certain time periods in the series~\cite{Zanin}, and to show that the degree of market inefficiency is correlated with the number of missing patterns~\cite{Zunino}. In the related setting of 
Gaussian processes with stationary increments, the frequency of ordinal patterns has been estimated in~\cite{SK} and computed exactly for some small patterns in~\cite{BS}.

Permutation patterns are found in a time series by taking $n$ data points at evenly-spaced time intervals and mapping them to a length-$n$ permutation determined by relative ordering.  For example, a sequence $4.8,-4.1,3.1,5.2$ would map to the permutation $3124 \in S_4$.  The frequencies of the patterns that occur are measured and used to make conclusions about the behavior of the data.  Central to this analysis is an understanding of the frequency with which patterns occur in a random time series.

Among the different models that are used for random time series, one of the most basic and applicable is a one-dimensional random walk in discrete time.  To construct such a walk, take $n-1$ independent and identically distributed (i.i.d.) random variables $X_1,X_2,\ldots,X_{n-1}$; we call these \emph{steps}.  At time $0$, the walker is at $0$ and at time $i$ the walker is at $X_1+X_2+\ldots+X_i$.  It is easy to see that not all permutation patterns occur with equal probability in such a random walk. For example, if the $X_i$'s are chosen from a distribution that only takes positive values, then the pattern $123\ldots n$ will occur with probability~1.

Define a map $p:{\RR}^{n-1} \rightarrow S_n$ where $p(X_1,X_2,\ldots,X_{n-1})=\pi$ if the entries of the permutation $\pi$ have the same relative ordering as the walk $Z_0, Z_1, Z_2, \ldots, Z_{n-1}$, where $Z_i=X_1+\dots+X_i$ for all $i$ (with the convention that $Z_0=0$).
More precisely, $p(X_1,X_2,\ldots,X_{n-1})=\pi$ if $\pi(i)=| \{k \mid Z_k \le Z_i \}|$ for all $i$.  If the associated random walk contains repeated values, i.e. there exist some $i,j$ such that $Z_i=Z_j$, it will be our convention to leave $p$ undefined.  Since we only deal with continuous distributions, the probability of having repeated values is zero.

While not all permutations occur in the image of $p$ with equal probability, it turns out that there are certain classes of permutations that do have equal probability of occurring, regardless of the probability distribution chosen for the $X_i$'s.  For example, the pattern $132$ will always occur with equal probability as $213$, since $p(X_1,X_2)=132$ if and only if $p(X_2,X_1)=213$.  In general, the reverse-complement of a permutation will occur with equal probability as the permutation itself.  It turns out that such equivalencies are not restricted to reverse-complements, but are littered across~$S_n$. For example, $1432$ and $2143$ also occur with equal probability because $p(X_1,X_2,X_3)=1432$ if and only if $p(X_2,X_1,X_3)=2143$.

We are interested in equivalence classes of permutations for which any two patterns $\pi$ and $\tau$ in the same class have $\PPP(p(X_1,X_2,\ldots,X_{n-1})=\pi)=\PPP(p(X_1,X_2,\ldots,X_{n-1})=\tau)$, for any continuous probability distribution on the random variables $X_i$ (with the only requirement that they are i.i.d.).  

In this paper we define a natural equivalence relation on permutations that satisfies this property, and we completely characterize the corresponding equivalence classes. The main theorem is stated in Section~\ref{Equivalence}.  Sections~\ref{Valid Flips}, \ref{Irreducible Intervals}, and~\ref{Cohesive Intervals and Partitions} introduce the tools and ideas used in the proof. Finally, in Section \ref{Future Work} we state a conjecture that would strengthen our results. Some of the proofs are omitted in this extended abstract, and will appear in a future version of the paper.

 \section{Equivalence on Permutations} \label{Equivalence}
 
We define a natural equivalence relation on permutations $\pi, \tau \in S_n$, which is suggested by the above examples.  We let $\pi \sim \tau$ if there exists some $\rho \in S_{n-1}$ such that $p(x_1,x_2,\ldots,x_{n-1})=\pi$ if and only if $p(x_{\rho(1)},x_{\rho(2)},\ldots,x_{\rho(n-1)})=\tau$, for every $x_1,\dots,x_n\in\mathbb{R}$.  If $X_1,X_2,\ldots,X_{n-1}$ are i.i.d. random variables, then  the sequences $(X_1,\ldots,X_{n-1})$ and $(X_{\rho(1)},X_{\rho(2)},\ldots,X_{\rho(n-1)})$ have the same joint probability distribution, and so $\pi \sim \tau$ implies that $\PPP(p(X_1,X_2,\ldots,X_{n-1})=\pi)=\PPP(p(X_1,X_2,\ldots,X_{n-1})=\tau)$ for any continuous probability distribution on the random variables~$X_i$.
 
  The main result of this paper precisely characterizes the equivalence classes for $\sim$.  Our characterization is best illustrated by displaying permutations $\pi\in S_n$ on an $n\times n$ grid by filling the boxes $(i,\pi(i))$ for $1\le i\le n$ with a dot.  Our convention is that the $(i,j)$ box is in the $i$th column from the left and the $j$th row from the bottom, as in cartesian coordinates.

To state the main theorem, a few definitions need to be introduced.  A \emph{block} in a permutation is a set of consecutive entries whose values also form a consecutive set.  On the grid, a block is a square subgrid with a dot in each row and column, which implies that the regions right, left, above, and below a block are empty.  A \emph{cylindrical block} is a generalization of this notion where the requirement of consecutive positions is relaxed by considering $1$ and $n$ to be consecutive. If we identify the left and right edges of our $n\times n$ grid, then a cylindrical block is a square subgrid of the resulting cylinder with a dot in each row and column. Note that a cylindrical block can be either a regular block or a block that spans the left and right sides of the grid.  We say that a cylindrical block is \emph{bordered} if the entries in the block with highest and lowest value occur precisely at the outer positions (see Figure~\ref{CylEx}).

\begin{figure}[htb]
	\centering
	\begin{subfigure}[b]{0.275\textwidth} 
		\centering
		\begin{sideways}
		\begin{minipage}{3cm}
		\begin{tikzpicture} [scale=.35,every node/.style={draw,shape=circle,fill=black,scale=.5}]
			\draw[gray] (0,0) grid (10,10);
			\draw[fill=lightgray] (0,7) rectangle (1,8);
			\draw[fill=lightgray] (4,3) rectangle (5,4);
			\draw[very thick] (0,3) rectangle (5,8);
			\draw (0.5,7.5) node {} ;
			\draw (1.5,4.5) node {} ;
			\draw (2.5,6.5) node {} ;
			\draw (3.5,5.5) node {} ;
			\draw (4.5,3.5) node {} ;
			\draw (5.5,8.5) node {} ;
			\draw (6.5,1.5) node {} ;
			\draw (7.5,9.5) node {} ;
			\draw (8.5,0.5) node {} ;
			\draw (9.5,2.5) node {} ;
			\draw[black] (4.3,-.2)--(4.5,0)--(4.3,.2);
			\draw[black] (4.3,9.8)--(4.5,10)--(4.3,10.2);
			\draw[black] (4.5,-.2)--(4.7,0)--(4.5,.2);
			\draw[black] (4.5,9.8)--(4.7,10)--(4.5,10.2);
		\end{tikzpicture}
		\end{minipage}
		\end{sideways}
		\caption{\small The bordered cylindrical block $1\,3\,4\,2\,5$.}
	\end{subfigure}
	\hspace{0.05\textwidth}
	\begin{subfigure}[b]{0.275\textwidth} 
		\centering
		\begin{sideways}
		\begin{minipage}{3cm}
		\begin{tikzpicture} [scale=.35,every node/.style={draw,shape=circle,fill=black,scale=.5}]
			\draw[gray] (0,0) grid (10,10);
			\draw[very thick] (4,0)--(4,4)--(10,4)--(10,0);
			\draw[very thick] (4,10)--(4,8)--(10,8)--(10,10);
			\draw (0.5,7.5) node {} ;
			\draw (1.5,4.5) node {} ;
			\draw (2.5,6.5) node {} ;
			\draw (3.5,5.5) node {} ;
			\draw (4.5,3.5) node {} ;
			\draw (5.5,8.5) node {} ;
			\draw (6.5,1.5) node {} ;
			\draw (7.5,9.5) node {} ;
			\draw (8.5,0.5) node {} ;
			\draw (9.5,2.5) node {} ;
			\draw[black] (4.3,-.2)--(4.5,0)--(4.3,.2);
			\draw[black] (4.3,9.8)--(4.5,10)--(4.3,10.2);
			\draw[black] (4.5,-.2)--(4.7,0)--(4.5,.2);
			\draw[black] (4.5,9.8)--(4.7,10)--(4.5,10.2);
		\end{tikzpicture}
		\end{minipage}
		\end{sideways}
		\caption{\small The unbordered cylindrical block $5\,10\,7\,9\,8\,6$.} \label{CylBlock}
    	\end{subfigure}
	\hspace{0.05\textwidth}
	\begin{subfigure}[b]{0.275\textwidth} 
		\centering
		\begin{sideways}
		\begin{minipage}{3cm}
			\begin{tikzpicture} [scale=.35,every node/.style={draw,shape=circle,fill=black,scale=.5}]
			\draw[gray] (0,0) grid (10,10);
			\draw[fill=lightgray] (5,8) rectangle (6,9);
			\draw[fill=lightgray] (9,2) rectangle (10,3);
			\draw[very thick] (5,0)--(5,3)--(10,3)--(10,0);
			\draw[very thick] (5,10)--(5,8)--(10,8)--(10,10);
			\draw (0.5,7.5) node {} ;
			\draw (1.5,4.5) node {} ;
			\draw (2.5,6.5) node {} ;
			\draw (3.5,5.5) node {} ;
			\draw (4.5,3.5) node {} ;
			\draw (5.5,8.5) node {} ;
			\draw (6.5,1.5) node {} ;
			\draw (7.5,9.5) node {} ;
			\draw (8.5,0.5) node {} ;
			\draw (9.5,2.5) node {} ;
			\draw[black] (4.3,-.2)--(4.5,0)--(4.3,.2);
			\draw[black] (4.3,9.8)--(4.5,10)--(4.3,10.2);
			\draw[black] (4.5,-.2)--(4.7,0)--(4.5,.2);
			\draw[black] (4.5,9.8)--(4.7,10)--(4.5,10.2);
		\end{tikzpicture}
		\end{minipage}
		\end{sideways}
		\caption{\small The bordered cylindrical block $10\,7\,9\,8\,6$.} 
	\end{subfigure}
	\caption{Examples of bordered and unbordered cylindrical blocks in $\pi=8\,6\,1\,3\,4\,2\,5\,10\,7\,9$.} \label{CylEx}
\end{figure}
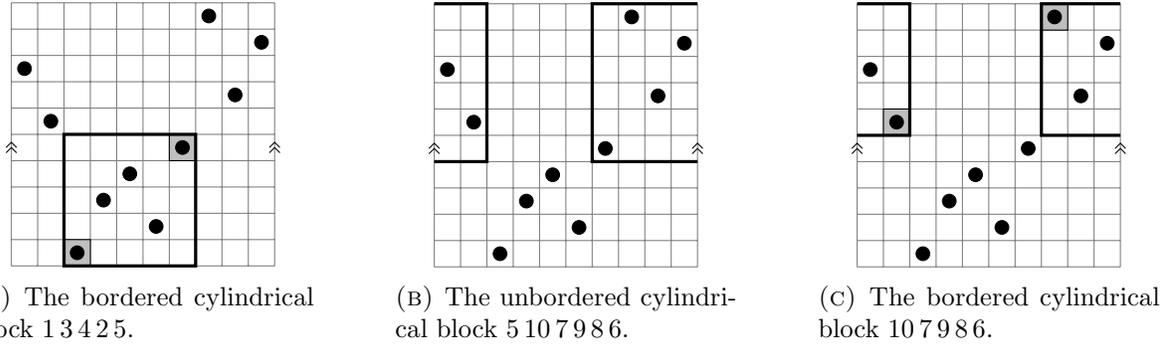

Given a permutation that contains a bordered cylindrical block, we can generate another permutation by performing a \emph{flip} on the bordered cylindrical block.  For a regular block, a flip is simply a 180$^\circ$ rotation of the contents of the block.  For a cylindrical block that spans the left/right boundary, a flip is akin to a 180$^\circ$ rotation, except the entries on the left side of the block are rotated and moved to the right and vice-versa, while the entries that are not part of the block are shifted right or left accordingly (see Figure~\ref{thm example}). Recall that the reverse-complement operation on $\pi$ corresponds to a 180$^\circ$ rotation of the whole $n\times n$ grid. Our main theorem characterizes equivalence classes in $S_n$ in terms of flips: 

\begin{thm} \label{Main Thm}
Let $\pi, \tau \in S_n$.  Then $\pi \sim \tau$ if and only if $\tau$ can be obtained from $\pi$ through a sequence of flips of bordered cylindrical blocks and the reverse-complement operation.
\end{thm}

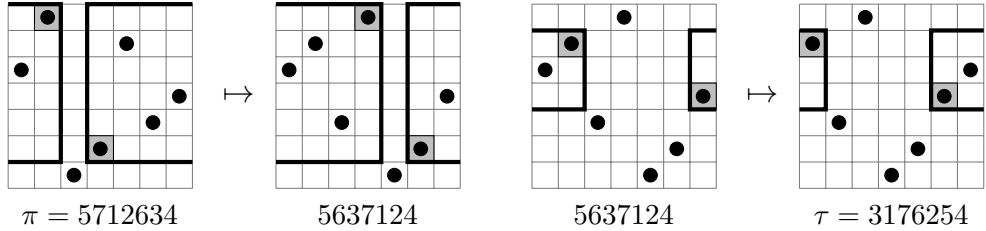
\begin{figure}[htb]
\centering
	\begin{tikzpicture} [scale=.35]
		\draw[gray] (0,0) grid (7,7);
		\draw[fill=lightgray] (1,6) rectangle (2,7);
		\draw[fill=lightgray] (3,1) rectangle (4,2);
		\draw[ultra thick] (7,1)--(3,1)--(3,7)--(7,7);
		\draw[ultra thick] (0,1)--(2,1)--(2,7)--(0,7);
		\draw (0.5,4.5) node[Bullet] {} ;
		\draw (1.5,6.5) node[Bullet] {} ;
		\draw (2.5,0.5) node[Bullet] {} ;
		\draw (3.5,1.5) node[Bullet] {} ;
		\draw (4.5,5.5) node[Bullet] {} ;
		\draw (5.5,2.5) node[Bullet] {} ;
		\draw (6.5,3.5) node[Bullet] {} ;
		\draw (3.5,-1) node{$\pi=5712634$};
		\draw (8.75,3.5) node{\large$\mapsto$};
	\end{tikzpicture}
	\begin{tikzpicture} [scale=.35]
		\draw[gray] (0,0) grid (7,7);
		\draw[fill=lightgray] (3,6) rectangle (4,7);
		\draw[fill=lightgray] (5,1) rectangle (6,2);
		\draw[ultra thick] (7,1)--(5,1)--(5,7)--(7,7);
		\draw[ultra thick] (0,1)--(4,1)--(4,7)--(0,7);
		\draw (0.5,4.5) node[Bullet] {} ;
		\draw (1.5,5.5) node[Bullet] {} ;
		\draw (2.5,2.5) node[Bullet] {} ;
		\draw (3.5,6.5) node[Bullet] {} ;
		\draw (4.5,0.5) node[Bullet] {} ;
		\draw (5.5,1.5) node[Bullet] {} ;
		\draw (6.5,3.5) node[Bullet] {} ;
		\draw (3.5,-1) node{$5637124$};
	\end{tikzpicture} \quad \quad
	\begin{tikzpicture} [scale=.35]
		\draw[gray] (0,0) grid (7,7);
		\draw[fill=lightgray] (1,5) rectangle (2,6);
		\draw[fill=lightgray] (6,3) rectangle (7,4);
		\draw[ultra thick] (0,3)--(2,3)--(2,6)--(0,6);
		\draw[ultra thick] (7,3)--(6,3)--(6,6)--(7,6);
		\draw (0.5,4.5) node[Bullet] {} ;
		\draw (1.5,5.5) node[Bullet] {} ;
		\draw (2.5,2.5) node[Bullet] {} ;
		\draw (3.5,6.5) node[Bullet] {} ;
		\draw (4.5,0.5) node[Bullet] {} ;
		\draw (5.5,1.5) node[Bullet] {} ;
		\draw (6.5,3.5) node[Bullet] {} ;
		\draw (8.75,3.5) node{\large$\mapsto$};
		\draw (3.5,-1) node{$5637124$};
	\end{tikzpicture}
	\begin{tikzpicture} [scale=.35]
		\draw[gray] (0,0) grid (7,7);
		\draw[fill=lightgray] (0,5) rectangle (1,6);
		\draw[fill=lightgray] (5,3) rectangle (6,4);
		\draw[ultra thick] (0,3)--(1,3)--(1,6)--(0,6);
		\draw[ultra thick] (7,3)--(5,3)--(5,6)--(7,6);
		\draw (0.5,5.5) node[Bullet] {} ;
		\draw (1.5,2.5) node[Bullet] {} ;
		\draw (2.5,6.5) node[Bullet] {} ;
		\draw (3.5,0.5) node[Bullet] {} ;
		\draw (4.5,1.5) node[Bullet] {} ;
		\draw (5.5,3.5) node[Bullet] {} ;
		\draw (6.5,4.5) node[Bullet] {} ;
		\draw (3.5,-1) node{$\tau=3176254$};
		\end{tikzpicture}
	\caption{A sequence of flips of bordered cylindrical blocks that maps $\pi=5712634$ to $\tau=6371245$.} 
	\label{thm example}
\end{figure}
 

Proving this result requires us to delve into the structure of permutations.  Our aim is to understand how we can permute the steps of one permutation to get another.  We define a mapping $L:S_n \rightarrow GL_{n-1}(\mathbb{C})$ that encodes the structure of~$\pi$.  Let $L(\pi)$ be the $(n-1)\times(n-1)$ matrix with entries $-1,0,1$ where \[(L(\pi))_{ij}=\begin{cases} \sgn(\pi(i+1)-\pi(i))), & \mbox{if } \pi(i) \leq j < \pi(i+1) \mbox{ or } \pi(i+1) \leq j < \pi(i),\\ 0, & \mbox{otherwise.}  \end{cases}\]

For example, $$L(32541)= \left[ \begin{array}{cccc} 0 & -1 & 0 & 0 \\ 0 & 1 & 1 & 1 \\ 0 & 0 & 0 & -1 \\ -1 & -1 & -1 & 0 \end{array} \right].$$

The following lemma states that $L$ is a group homomorphism. Thus, 
$L$ gives a representation of the symmetric group, which can be shown to be isomorphic to the standard representation.

\begin{lem} \label{rep}
For every $\pi, \tau \in S_n$, we have $L(\tau \pi)=L(\pi)L(\tau)$.
\end{lem}

It follows from Lemma~\ref{rep} that $L(\pi^{-1})=L(\pi)^{-1}$, and, in particular, that $L(\pi)$ is always invertible.  Additionally, since $L(\pi)$ and $L(\pi^{-1})$ are matrices with integral entries, they have integral determinants, and so $\det(L(\pi))=\pm1$.

Before proving Lemma~\ref{rep}, it is helpful to develop some intuition and terminology concerning the structure that is encoded in $L(\pi)$.  For real numbers $x_i$ to satisfy $p(x_1,x_2,\ldots,x_{n-1})=\pi$, there are some forced relationships among them.  For example, if $p(x_1,x_2,x_3)=1423$, then
$x_1+x_2>0$ and $x_1>x_3$, among other relations.



It will be convenient to draw each row of $L(\pi)$ as a vertical, directed edge. In the coordinate plane, we draw edge $e_i$ as the line segment that connects $(i, \pi(i))$ to $(i, \pi(i+1))$, where the edge is directed upwards if $\pi(i+1)>\pi(i)$ and downwards if $\pi(i)>\pi(i+1)$. Define $\sgn(e_i)=\sin(\pi(i+1)-\pi(i))$.  We denote the set of $n-1$ edges corresponding to $\pi$ by $\EE_{\pi}$ and call it its \emph{edge diagram}.  An example is shown in Figure~\ref{Level Ex}. We think of the $y$-coordinates $1,2,\ldots,n$ as vertices, and so we consider $e_i$ as a directed edge from vertex $\pi(i)$ to vertex $\pi(i+1)$.  With this interpretation, a sequence of edges, $e_i,e_{i+1},\ldots,e_{j-1}$, forms a path from vertex $\pi(i)$ to vertex $\pi(j)$.
 
 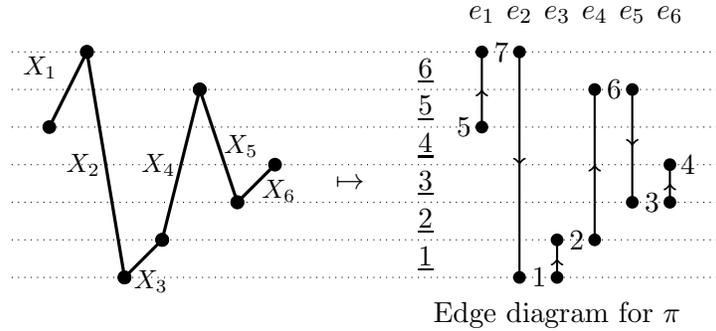
\begin{figure}[htb]
	\centering
		\begin{tikzpicture}[scale=.5]
			\foreach \x in {1,2,...,7}
			\draw[dotted] (0,\x)--(19,\x);
			\draw[very thick,*--*] (1,5)--(1.5,6) node[above left] {\small $X_1$}--(2,7);
			\draw[very thick, *--*] (2,7)--(2.5,4) node[xshift=-.3cm] {\small $X_2$}--(3,1);
			\draw[very thick, *--*] (3,1)--(3.5,1.5) node[yshift=-.3cm, xshift=.1cm] {\small $X_3$}--(4,2);
			\draw[very thick, *--*] (4,2)--(4.5,4) node[xshift=-.3cm]{\small $X_4$}--(5,6);
			\draw[very thick, *--*] (5,6)--(5.5,4.5) node[xshift=.3cm]{\small $X_5$}--(6,3);
			\draw[very thick, *--*] (6,3)--(6.5,3.5) node[xshift=.3cm,yshift=-.1cm]{\small $X_6$}--(7,4);

			\draw(9,3.5) node{$\mapsto$};


			\draw[thick,*->-*](12.5,5) node[left]{5}--(12.5,7);
			\draw[thick,*->-*](13.5,7) node[left]{7}--(13.5,1);
			\draw[thick,*->-*](14.5,1) node[left]{1}--(14.5,2);
			\draw[thick,*->-*](15.5,2) node[left]{2}--(15.5,6);
			\draw[thick,*->-*](16.5,6) node[left]{6}--(16.5,3);
			\draw[thick,*->-*](17.5,3) node[left]{3}--(17.5,4) node[right]{4};
			{\foreach \x in {1,2,...,6}
			\draw (11.5+\x,8) node{$e_{\x}$};}
			{\foreach \x in {1,2,...,6}
			\draw (11,0.5+\x) node{$\ubar{\x}$};
			}
			\draw (14.5,0) node {Edge diagram for $\pi$};
			
		\end{tikzpicture}
		\caption{The permutation $\pi=5712634$ as an edge diagram.}
		\label{Level Ex}
 \end{figure}
 
We partition the $y$-axis of our edge diagram into the intervals $\ubar{1}=[1,2]$, $\ubar{2}=[2,3]$, \ldots, $\ubar{n-1}=[n-1,n]$, that we call \emph{levels}.  We can then write each edge as a formal sum of the levels it covers: $e_i=\sum_{j=\pi(i)} ^{\pi(i+1)-1} \ubar{j}$ when $\pi(i)<\pi(i+1)$ and $e_i=-\sum_{j=\pi(i+1)} ^{\pi(i)-1} \ubar{j}$ when $\pi(i+1)<\pi(i)$.  Notice that this definition is equivalent to setting $e_i=\sum_{j=1}^{n-1} (L(\pi))_{ij} \cdot \ubar{j}$, that is,
\begin{equation} \label{edges-levels} \left[ \begin{array}{c} e_1 \\ e_2 \\ \vdots \\ e_{n-1} \end{array} \right] = L(\pi) \cdot \left[ \begin{array}{c} \ubar{1} \\ \ubar{2} \\ \vdots \\ \ubar{n-1} \end{array} \right]. \end{equation}
 For ease of notation, we write $\ubar{j} \in e_i$ if $(L(\pi))_{ij}\neq 0$.
We will use the notation $e_i^{\pi}$ if the permutation needs to be specified.

\begin{proof}[Proof of Lemma~\ref{rep}]
By Equation~\eqref{edges-levels}, the rows of $L(\pi)$ express the edges $e^\pi_i$ in terms of the levels~$\ubar{j}$ as $e^\pi_i=\sum_{j=\pi(i)} ^{\pi(i+1)-1} \ubar{j}$ if $\pi(i)<\pi(i+1)$ or $e^\pi_i=-\sum_{j=\pi(i+1)} ^{\pi(i)-1} \ubar{j}$ if $\pi(i+1)<\pi(i)$, and similarly for $L(\tau)$. Thus, the $i$-th row of the product $L(\pi)L(\tau)$ gives the expression of $\sum_{j=\pi(i)} ^{\pi(i+1)-1} e^\tau_j$ 
or $-\sum_{j=\pi(i+1)} ^{\pi(i)-1} e^\tau_j$ in terms of the levels.

On the other hand, note that the edges in the sums $\sum_{j=\pi(i)}^{\pi(i+1)-1} e^\tau_j$ and $-\sum_{j=\pi(i+1)} ^{\pi(i)-1} e^\tau_j$ form a path connecting $\tau(\pi(i))$ and $\tau(\pi(i+1))$ in the edge diagram of $\tau$. It follows that these sums equal $\sum_{k=\tau(\pi(i))}^
{\tau(\pi(i+1))-1} \ubar{k}$ or $-\sum_{k=\tau(\pi(i+1))}^{\tau(\pi(i))-1} \ubar{k}$, depending on the sign of $\tau(\pi(i))-\tau(\pi(i+1))$. The coefficients of these sums are precisely the $i$-th row of $L(\tau\pi)$ by definition, so we have shown that the $i$-th rows of $L(\pi)L(\tau)$ and $L(\tau\pi)$ are equal.
\end{proof}

In order to classify equivalence classes, we first show that $\pi \sim \tau$ if and only if $L(\pi)$ and $L(\tau)$ are related by permutations of rows and columns.  In the following lemma, $P_{\rho}$ denotes the permutation matrix associated to $\rho$; that is, $(P_{\rho})_{ij}$ equals $1$ if $\rho(i)=j$ and $0$ otherwise.

\begin{lem} \label{permutation}
For $\pi, \tau \in S_n$, $\pi \sim \tau$ if and only if there exist $\sigma, \rho \in S_{n-1}$ such that $P_{\rho^{-1}} L(\pi) P_{\sigma}=L(\tau)$.
\end{lem}

The proof of this lemma uses the fact that if we let $D_{\pi}=\{\xx \in \RR^{n-1} \mid p(\xx)=\pi \}$, then multiplication by $L(\pi)$ is a bijection between $\RR^{n-1}_{>0}$ and $D_{\pi}$.
Thus, if $\rho\in S_{n-1}$ is such that $(x_1,x_2,\ldots,x_{n-1}) \in D_{\pi}$ if and only if $(x_{\rho(1)},x_{\rho(2)}, \ldots,x_{\rho(n-1)}) \in D_{\tau}$, then multiplication by
$L(\pi^{-1}) P_{\rho} L(\tau)$ is a bijection from $\RR^{n-1}_{>0}$ to itself. Since this matrix and its inverse have non-negative integer entries, it must be a permutation matrix $P_\sigma$.

If we consider the results of Lemma~\ref{permutation} in the context of an edge diagram, $\sigma$ is applied to the levels and $\rho$ is applied to the edges.  In an edge diagram, we will only need to consider permutations of the levels by $\sigma$, as will be stated in Lemma~\ref{level perm}.
 
We need a notation that describes edges in terms of levels.  For vertices $s,t \in [n]$ with $s<t$ in an edge diagram, we write the \emph{interval} between $s$ and $t$ as $[s,t]$ and define this to be the union of levels $\ubar{s} \cup \ubar{s+1} \cup \ldots \cup \ubar{t-1}$.  We write an edge $e_i \in \EE_\pi$ as a directed interval $[\pi(i),\pi(i+1)]_+$ (if $\pi(i)<\pi(i+1)$) or $[\pi(i+1),\pi(i)]_-$ (if $\pi(i)>\pi(i+1)$) for $1\le i\le n-1$; the subscripts indicate upwards or downwards direction, respectively.  For example, in Figure \ref{Level Ex}, the edges in $\EE_{\pi}$ are $e_1=[5,7]_+,e_2=[1,7]_-, e_3=[1,2]_+,e_4=[2,6]_+, e_5=[6,3]_-, e_6=[3,4]_+$.  
Given an edge $e=[i,j]_{\pm} \in \EE_{\pi}$ and any interval $[s,t]$, we write $[s,t] \subseteq e$ if $[s,t] \subseteq [i,j]$ and $e \subseteq [s,t]$ if $[i,j] \subseteq [s,t]$.

Permuting the levels of an edge diagram by $\sigma$ takes level $\ubar{i}$ and moves it to height $\ubar{\sigma(i)}$.  Intervals and edges are shifted accordingly: the interval $[i,j]$ is moved to $\sigma.[i,j]=\ubar{\sigma(i)} \cup \ubar{\sigma(i+1)} \cup \ldots \cup \ubar{\sigma(j-1)}$ (note that $\sigma.[i,j]$ may no longer be an interval), and the edge $[i,j]_{\pm}$ is moved to $(\sigma.[i,j])_{\pm}$, where the sign is preserved.  The set of images of elements of $\EE_{\pi}$ is denoted by
$\EE_{\sigma.\pi}$, and we say that this is a \emph{well-defined} edge diagram if for every edge $[i,j]_{\pm} \in E_{\pi}$, $\sigma.[i,j]$ is an interval.  If $\EE_{\sigma.\pi}$ is the edge diagram of a permutation, we say that $\EE_{\sigma.\pi}$ is a \emph{proper} edge diagram and call the corresponding permutation $\sigma.\pi$.

Recall that the edge diagram of a permutation $\tau$ forms a path $\tau(1), \tau(2), \ldots, \tau(n)$ where the vertices $\tau(i)$ and $\tau(i+1)$ are connected by an edge $e_i$. A well-defined edge diagram $\EE_{\sigma.\pi}$ is proper if and only if its edges form a path.  In general, we do not consider the edges of the set $\EE_{\pi}$ ordered.  This idea is useful when proving the following statement:
 
\begin{lem} \label{level perm}
Given $\pi, \tau  \in S_n$, $\pi \sim \tau$ if and only if there exists $\sigma \in S_{n-1}$ such that $\EE_{\sigma.\pi}=\EE_{\tau}$.
\end{lem}

\omitt{
\begin{proof}
First assume $\pi \sim \tau$.  Then there exist some $\sigma, \rho \in S_{n-1}$ such that $P_{\rho^{-1}} L(\pi) P_{\sigma}=L(\tau)$.  So $L(\pi)P_{\sigma}$ is a matrix whose rows are those of $L(\tau)$ in a different order.  Consider $\EE_{\sigma.\pi}$.   Recall that the rows of $L(\pi)$ and $L(\tau)$ correspond to edges of the edge diagram, while the columns correspond to the levels.  Since edges in edges diagrams are not ordered, it follows that the set $\EE_{\sigma.\pi}$ and $\EE_{\tau}$ are equal.

Now assume $\EE_{\sigma.\pi}=\EE_{\tau}$.  Consider a labeling of the edges of $\EE_{\pi}$ where $e_i^{\pi}=[\pi(i),\pi(i+1)]_+$ or $[\pi(i+1),\pi(i)]_-$.  Then there exists some permutation $\rho \in S_{n-1}$ such that $\sigma.e_i^{\pi}=e_{\rho(i)}^{\tau}$ for every $1 \leq i \leq n-1$.  This implies that $ L(\pi) P_{\sigma}=P_{\rho}L(\tau)$ and therefore $\pi \sim \tau$.
\end{proof}}

 \begin{example}
Let $\pi=54621873,\tau=73218463\in S_8$, whose edge diagrams are drawn in Figure~\ref{actionex}, and let $\sigma = 2365471\in S_7$. Then $\EE_{\pi}=\{[4,5]_-,[4,6]_+,[2,6]_-,[1,2]_-,[1,8]_+,[7,8]_-,[3,7]_-\}$ and $\EE_{\sigma.\pi} = \{[3,7]_-,[2,3]_-,[1,2]_-,[1,8]_+,[4,8]_-,[4,6]_+,[6,5]_- \}=\EE_{\tau}$.  Therefore $\pi \sim \tau$.
 \end{example}
 
	
	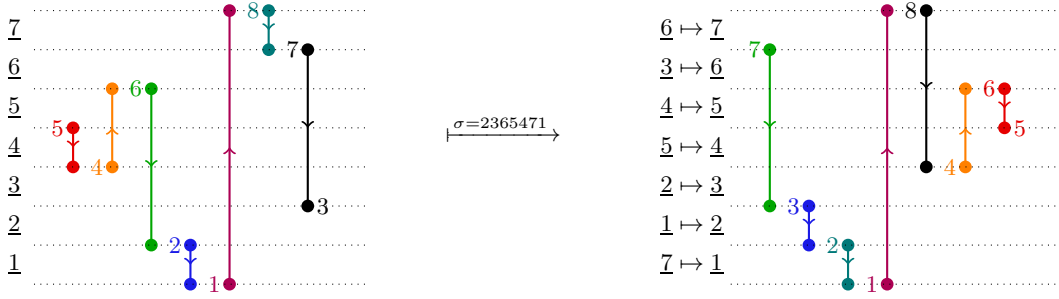
\begin{figure}[h]
\footnotesize\centering
	\begin{tikzpicture} [scale=.52]
	
		\draw[thick,*->-*, black!10!red](1,5) node[left]{5}--(1,4);
		\draw[thick,*->-*, orange](2,4) node[left]{4}--(2,6);
		\draw[thick,*->-*, black!35!green](3,6) node[left]{6}--(3,2);
		\draw[thick,*->-*, gray!20!blue](4,2) node[left]{2}--(4,1);
		\draw[thick,*->-*, blue!10!purple](5,1) node[left]{1}--(5,8);
		\draw[thick,*->-*, black!5!teal](6,8) node[left]{8}--(6,7);
		\draw[thick,*->-*](7,7) node[left]{7}--(7,3)node[right]{3};
		{\foreach \x in {1,2,...,8}
		\draw[dotted] (0,\x)--(8.5,\x);}
		{\foreach \x in {1,2,...,7}
		\draw (-.5,\x+.5) node{$\ubar{\x}$};}

		\draw (12,5) node{$\xmapsto{\sigma=2365471}$};				
	\end{tikzpicture} \hskip 10mm
	\begin{tikzpicture} [scale=.52]
	
		\draw[thick,*->-*,black!35!green](1,7) node[left]{7}--(1,3);
		\draw[thick,*->-*,gray!20!blue](2,3) node[left]{3}--(2,2);
		\draw[thick,*->-*,black!5!teal](3,2) node[left]{2}--(3,1);
		\draw[thick,*->-*, blue!10!purple](4,1) node[left]{1}--(4,8);
		\draw[thick,*->-*](5,8) node[left]{8}--(5,4);
		\draw[thick,*->-*,orange](6,4) node[left]{4}--(6,6);
		\draw[thick,*->-*,black!10!red](7,6) node[left]{6}--(7,5)node[right]{5};
		{\foreach \x in {1,2,...,8}
		\draw[dotted] (0,\x)--(8.5,\x);}

		\draw (-1,2.5) node{$\ubar{1}\mapsto \ubar{2}$};
		\draw (-1,3.5) node{$\ubar{2}\mapsto \ubar{3}$};
		\draw (-1,6.5) node{$\ubar{3}\mapsto \ubar{6}$};
		\draw (-1,5.5) node{$\ubar{4}\mapsto \ubar{5}$};
		\draw (-1,4.5) node{$\ubar{5}\mapsto \ubar{4}$};
		\draw (-1,7.5) node{$\ubar{6}\mapsto \ubar{7}$};
		\draw (-1,1.5) node{$\ubar{7} \mapsto \ubar{1}$};
						
	\end{tikzpicture}
	\caption{$\pi=54621873$ is mapped to $\tau=73218463$ by applying $\sigma=2365471$. The edge $[3,7]_-$ is mapped to $\sigma.[3,7]_-=[4,8]_-$. } \label{actionex}
\end{figure}

\begin{definition}
Let $\pi \in S_n$ and $\sigma \in S_{n-1}$ such that $\EE_{\sigma.\pi}$ is well-defined.  We say that \emph{$\sigma$ is valid with respect to $\pi$} if $\EE_{\sigma.\pi}$ is proper.  If $\sigma$ is valid with respect to $\pi$, we will say that $\sigma$ acts validly on~$\pi$.
\end{definition}

One of the goals of the next few sections to describe all $\sigma$ that are valid with respect to $\pi$ and understand how they transform~$\pi$.

Note that if $\sigma_1, \sigma_2 \in S_{n-1}$ are such that $\EE_{\sigma_1.\pi}$ and $\EE_{\sigma_2.(\sigma_1.\pi)}$ are proper edge diagrams, then $\EE_{\sigma_2.(\sigma_1.\pi)}=\EE_{(\sigma_2 \sigma_1).\pi}$. Allowing edge diagrams that are not well-defined (meaning the edges are not contiguous) would give a group action of $S_{n-1}$ on $S_n$, but for our purposes it is convenient to restrict only to well-defined diagrams.

\section{Valid Flips} \label{Valid Flips}

In this section, we define an operation on edge diagrams that is analogous to a flip of a bordered cylindrical block.  The remaining sections will then focus on proving that any two equivalent permutations differ by a sequence of these operations.

 We will use the following two properties of edge diagrams:

\begin{enumerate}[(1)]
\item  Let $\{p,q\} =\{\pi(1), \pi(n)\}$ be the endpoints in the edge diagram for $\pi$ (assume $p<q$).  Then for any $i \in [n-1]$, the number of edges that contain $\ubar{i}$ is even if $\ubar{i} \in [p,q]$ and odd otherwise.  So, if $\sigma$ is valid with respect to $\pi$, then $\sigma.[p,q]$ must be an interval.

\item Define a cycle in an edge diagram to be a sequence of vertices $v_1,v_2,v_3, \ldots, v_k$ where either $[v_i,v_{i+1}]_+$ or $[v_{i+1},v_i]_-$ is an edge for every $i$, and either $[v_k,v_{1}]_+$ or $[v_{1},v_k]_-$ is an edge.  If $\EE_{\sigma.\pi}$ is well-defined, then it does not contain a cycle.
\end{enumerate}

Property (1) is straightforward. To prove property (2), one can show that a cycle in $\EE_{\sigma.\pi}$ would require the existence of a cycle in $\EE_\pi$.

Define a \emph{flip} to be a permutation $\sigma \in S_{n-1}$ of the form $$\sigma =\bigl(\begin{smallmatrix} 1 & 2 & 3 & \cdots & i-1 & i & i+1 & \cdots & j-1 & j  & \cdots & n-1 \\ 1 & 2 & 3 & \cdots & i-1 & j-1 & j-2 & \cdots &  i & j & \cdots & n \end{smallmatrix} \bigr).$$  We say that $\sigma$ flips the interval $[i,j]$.

Since adjacent transpositions are a particular case of flips, it is clear that flips generate $S_{n-1}$.  Additionally, property (1) allows us to make a stronger statement.  Fix $\pi$, and let $F([p,q]) \subseteq S_{n-1}$ be the set of flips of intervals $[i,j]$ such that $[i,j] \subseteq [p,q]$, $[p,q] \subseteq [i,j]$, or $[p,q] \cap [i,j] = \emptyset$.  Then, for any $\sigma$ that is valid with respect to $\pi$, there exists a sequence of flips $\sigma_1,\sigma_2,\ldots,\sigma_{k}$ such that $\sigma_k \sigma_{k-1} \ldots \sigma_1 = \sigma$ and $\sigma_i \in F( \sigma_{i-1} \sigma_{i-2} \ldots \sigma_1.[p,q])$ for every $1 \leq i \leq k$.  Note that we are not claiming that $\EE_{\sigma_{i-1} \sigma_{i-2} \ldots \sigma_1.\pi}$ is proper or even well-defined for any $i$. 

\begin{lem} \label{valid flip}
Suppose that $\sigma \in F([p,q])$ flips the interval $[i,j]$ and that for every $e \in \EE_{\pi}$, we have $e \subseteq [i,j]$, $[i,j] \subseteq e$, or $[i,j] \cap e = \emptyset$.  Then $\sigma$ is valid with respect to $\pi$.
\end{lem}

We omit the proof of this lemma.  The idea is that the path in $\EE_\pi$ behaves well in relation to the interval $[i,j]$.  Therefore, flipping $[i,j]$ will simply reorder certain portions of the path.

It turns out that the intervals of the type described in Lemma \ref{valid flip} are precisely the tool we need to characterize the equivalence classes for $\sim$.  We therefore call a flip as described in Lemma \ref{valid flip} a \emph{valid flip} (with respect to $\pi$) and the interval it flips a \emph{valid interval} (in $\EE_\pi$). In general, we say that $\sigma$ transforms $\pi$ by a \emph{sequence of valid flips} if there exist $\sigma_1, \sigma_2, \ldots, \sigma_l \in S_{n-1}$ with $\sigma_l \sigma_{l-1} \cdots \sigma_1=\sigma$ such that $\sigma_i$ is a valid flip with respect to $\sigma_{i-1} \sigma_{i-2} \ldots \sigma_1 . \pi$ for every $i$. 

\begin{lem} \label{bordered=valid}
For every $\pi \in S_n$, the interval $[i,j]$ is valid in $\EE_\pi$ if and only if in the grid of $\pi$, the values $[i,j]$ (which are in positions $\pi^{-1}(i), \pi^{-1}(i+1), \ldots, \pi^{-1}(j)$) form a bordered cylindrical block. In this case, flipping the valid interval $[i,j]$ in $\EE_\pi$ is equivalent to flipping the corresponding bordered cylindrical block in the grid of $\pi$.
\end{lem}

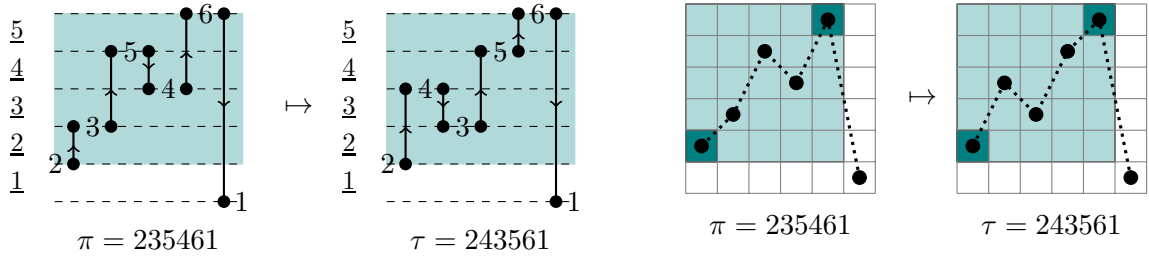
\begin{figure}[htb]
\centering
		\begin{tikzpicture} [scale=.5]
		\draw[fill=teal!30!white,teal!30!white] (.5,2) rectangle (5.5,6);
		\draw[thick,*->-*](1,2) node[left]{2}--(1,3);
		\draw[thick,*->-*](2,3) node[left]{3}--(2,5);
		\draw[thick,*->-*](3,5) node[left]{5}--(3,4);
		\draw[thick,*->-*](4,4) node[left]{4}--(4,6);
		\draw[thick,*->-*](5,6) node[left]{6}--(5,1) node[right]{1};
		{\foreach \x in {1,2,...,6}
		\draw[dashed] (.5,\x)--(5.5,\x);
		}
		{\foreach \x in {1,2,...,5}
		\draw (-.5,\x+.5) node{\small$\ubar{\x}$};
		}	
		\draw (3,0) node{$\pi=235461$};
		\draw (7,3.5) node{	$\mapsto$};			
	\end{tikzpicture}
	\begin{tikzpicture} [scale=.5]
		\draw[fill=teal!30!white,teal!30!white] (.5,2) rectangle (5.5,6);
		\draw[thick,*->-*](4,5) node[left]{5}--(4,6);
		\draw[thick,*->-*](3,3) node[left]{3}--(3,5);
		\draw[thick,*->-*](2,4) node[left]{4}--(2,3);
		\draw[thick,*->-*](1,2) node[left]{2}--(1,4);
		\draw[thick,*->-*](5,6) node[left]{6}--(5,1) node[right]{1};
		{\foreach \x in {1,2,...,6}
		\draw[dashed] (.5,\x)--(5.5,\x);
		}
		\draw(-.5,1.5) node{\small $\ubar{1}$};
		\draw(-.5,2.5) node{\small $\ubar{2}$};
		\draw(-.5,3.5) node{\small $\ubar{3}$};
		\draw(-.5,4.5) node{\small $\ubar{4}$};
		\draw(-.5,5.5) node{\small $\ubar{5}$};
		\draw (3,0) node{$\tau=243561$};
	\end{tikzpicture}
			\hspace{10mm}
	\begin{tikzpicture}[scale=.42]
		\draw[fill=teal!30!white] (0,1) rectangle (5,6);
		\draw[fill=teal] (0,1) rectangle (1,2);
		\draw[fill=teal] (4,5) rectangle (5,6);
		\draw[gray] (0,0) grid (6,6);
		\draw (0.5,1.5) node[Bullet] {} ;
		\draw (1.5,2.5) node[Bullet] {} ;
		\draw (2.5,4.5) node[Bullet] {} ;
		\draw (3.5,3.5) node[Bullet] {} ;
		\draw (4.5,5.5) node[Bullet] {} ;
		\draw (5.5,0.5) node[Bullet] {} ;
		\draw (3,-1) node{$\pi=235461$};
		\draw (7.5,3) node{$\mapsto$};	
		\draw[very thick, dotted] (.5,1.5)--(1.5,2.5)--(2.5,4.5)--(3.5,3.5)--(4.5,5.5)--(5.5,0.5);
				\draw (4,-1.75) node{};	
	\end{tikzpicture}
	\begin{tikzpicture}[scale=.42]
		\draw[fill=teal!30!white] (0,1) rectangle (5,6);
		\draw[fill=teal] (0,1) rectangle (1,2);
		\draw[fill=teal] (4,5) rectangle (5,6);
		\draw[gray] (0,0) grid (6,6);
		\draw (0.5,1.5) node[Bullet] {} ;
		\draw (1.5,3.5) node[Bullet] {} ;
		\draw (2.5,2.5) node[Bullet] {} ;
		\draw (3.5,4.5) node[Bullet] {} ;
		\draw (4.5,5.5) node[Bullet] {} ;
		\draw (5.5,0.5) node[Bullet] {} ;
		\draw (3,-1) node{$\tau=243561$};
		\draw[very thick, dotted] (.5,1.5)--(1.5,3.5)--(2.5,2.5)--(3.5,4.5)--(4.5,5.5)--(5.5,0.5);
		\draw (4,-1.75) node{};
	\end{tikzpicture}
\caption{Flipping valid interval $[2,6]$ in the edge diagram for $\pi=235461$ yields the same result as flipping the bordered cylindrical block with values $[2,6]$ (and positions $[1,5]$) in the grid of $\pi$.}
\label{bordered=valid fig}
\end{figure}

For an example of Lemma \ref{bordered=valid}, see Figure~\ref{bordered=valid fig}.  Now we can restate Theorem~\ref{Main Thm} as follows:

\begin{thm} \label{main theorem}
Let $\pi, \tau \in S_n$.  Then $\pi \sim \tau$ if and only if $\tau$ can be obtained from $\pi$ by a sequence of valid flips.
\end{thm}

In order to prove this theorem, we need to further explore the structure of edge diagrams.  The aim of the definitions and lemmas in the following section will be to decompose our edge diagrams into nested structures to which we will be able to apply an inductive argument.

\section{Irreducible Intervals} \label{Irreducible Intervals}
For given $\pi\in S_n$, most permutations of the levels of its edge diagram are not be valid.
Our induction argument will rely on the ability to partition the levels of the edge diagram into intervals that remain intervals under any valid permutation of the levels.  We first introduce the idea of an irreducible interval, which is a maximal interval whose levels remain adjacent under the action of any valid permutation.  Then we show that every edge diagram can be uniquely partitioned into intervals of this type.

\begin{definition} \label{IrredInts}
Let $\pi \in S_n$.  An interval $[s,t]$ in the edge diagram of $\pi$ is said to satisfy the \emph{linked conditions} if, for every $\sigma \in S_{n-1}$ that is valid with respect to $\pi$, the following are true:
\begin{enumerate}[(1)]
\item The image $\sigma.[s,t]$ is an interval.
\item If we let $i,j \in [n]$ be such that $\sigma.[s,t]=[i,j]$, then the ordered tuple $(\sigma(s),\sigma(s+1),\sigma(s+2),\ldots,\sigma(t-1))$ equals either $(i,i+1,\ldots,j-1)$ or $(j-1,j-2,\ldots,i)$.
\end{enumerate}

Additionally, $[s,t]$ is called an \emph{irreducible interval} if it is a maximal interval that satisfies the linked conditions (i.e. for any interval $[x,y]$ such that $[s,t] \subsetneq [x,y]$, $[x,y]$ does not satisfy the linked conditions).
\end{definition}

Note that any proper subinterval of an irreducible interval will satisfy the linked conditions but will fail the maximality condition.

\begin{lem} \label{irreducible}
For any $\pi \in S_n$, the interval $[1,n]$ in the edge diagram of $\pi$ can be uniquely partitioned into irreducible intervals.
\end{lem}

\begin{proof}
Suppose that $[s,t]$ and $[u,v]$ are irreducible intervals with a nonempty intersection.  By the maximality condition, $[s,t]$ cannot properly contain $[u,v]$ and vice versa.  Without loss of generality, assume that $s \leq u$ and $t \leq v$.  Then $[s,t] \cap [u,v]= [u,t]$.  Since both $[s,t]$ and $[u,v]$ satisfy the linked conditions, every $\sigma$ that is valid with respect to $\pi$ maps $[s,t]$ and $[u,v]$ to intervals.  Since these intervals have non-empty intersection, it follows that $\sigma$ maps $[s,t] \cup [u,v]=[s,v]$ to an interval. 
In fact, the second linked condition for $[s,t]$ and $[u,v]$ quickly implies that $[s,v]$ satisfies the condition as well.
By the maximality of irreducible intervals, this is only possible if $s=u$ and $t=v$.  Therefore, any two irreducible intervals are either disjoint or equal.

Since irreducible intervals of width 1 trivially satisfy the linked conditions, it is clear that every level of the edge diagram is contained in some irreducible interval.  The condition on maximality for irreducible intervals implies that the partition is unique.
\end{proof}

We let $\mathcal{I_{\pi}}=\{[x_i,x_{i+1}]:0\le i<k\}$, where $x_0=1$ and $x_k=n$, denote the partition of the edge diagram for $\pi$ into irreducible intervals. We call $\mathcal{I_{\pi}}$ the \emph{irreducible partition} of $\pi$, and we call the $x_i$'s its borders.

Conveniently, valid intervals of width greater than one are the union of adjacent irreducible intervals.  Indeed, given a valid interval $[i,j]$ in $\EE_{\pi}$, we know by Lemma~\ref{valid flip} that if $\sigma$ flips $[i,j]$, then $\EE_{\sigma.\pi}$ is a proper edge diagram.  When $j \neq i+1$, this implies that $i$ and $j$ are borders of the irreducible partition $\mathcal{I}_{\pi}$, since $\sigma.(\ubar{i-1} \cup \ubar{i})=\ubar{i-1} \cup \ubar{j-1}$ and $\sigma.(\ubar{j-1} \cup \ubar{j})=\ubar{i} \cup \ubar{j}$ are not intervals.

Finding irreducible intervals using Definition~\ref{IrredInts} is impractical. One consequence of Theorem~\ref{main theorem} will be that irreducible intervals are completely determined by the valid intervals.  It will follow that any interval $[a,b]$ satisfies the linked conditions if and only if $[a,b]$ is contained in or disjoint from every valid interval in $\EE_{\pi}$.  Thus, $x_i$ is a border of $\mathcal{I_{\pi}}$ if and only if $x_i$ is an endpoint of some valid interval.

Although a few permutations only have irreducible intervals of width one, such as $1\,2\,3 \ldots n$ and $1\,n\,(n-1)\ldots 2$, permutations often have wide irreducible intervals, as is the case whenever the edge diagram has an edge of width two.  In fact, many permutations have one single irreducible interval $[1,n]$.  Examples of irreducible intervals are given in Figure \ref{irredintex}.


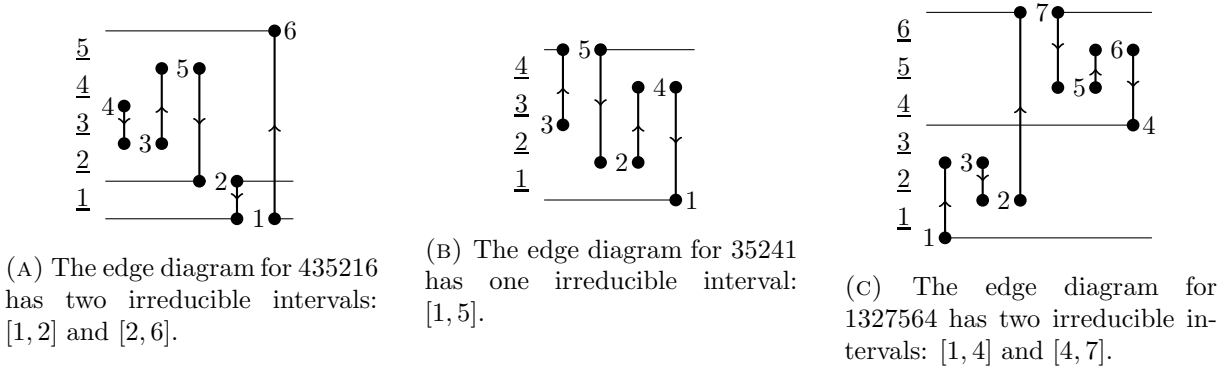
\begin{figure}[htb] 
	\centering
	\begin{subfigure}{.3\textwidth}
	\centering
	\small{
		\begin{tikzpicture} [scale=.5]
			\draw (.5,1)--(5.5,1);
			\draw (.5,2)--(5.5,2);		
			\draw (.5,6)--(5.5,6);
			\draw (1,4) node[left, rectangle,fill=white]{4};
			\draw (4,2) node[left, rectangle,fill=white]{2};
			\draw (5,1) node[left, rectangle,fill=white]{1};
			\draw (5,6) node[rectangle, fill=white, right]{6};
			\draw[thick,*->-*](1,4)--(1,3);
			\draw[thick,*->-*](2,3) node[left]{3}--(2,5);
			\draw[thick,*->-*](3,5) node[left]{5}--(3,2);
			\draw[thick,*->-*](4,2)--(4,1);
			\draw[thick,*->-*](5,1)--(5,6);
			\foreach \x in {1,2,...,5}
			\draw (-.1,\x-.5) node[yshift=.5cm]{$\ubar{\x}$};							
		\end{tikzpicture} }
		\caption{The edge diagram for $435216$ has two irreducible intervals: $[1,2]$ and $[2,6]$.}
	\end{subfigure} \hskip 5mm
	\begin{subfigure}{.3\textwidth} 
	\centering
	\small{
		\begin{tikzpicture} [scale=.5]
			\draw (.5,1)--(4.5,1);		
			\draw (.5,5)--(4.5,5);
			\draw (1,3) node[left, rectangle,fill=white]{3};
			\draw (2,5) node[left, rectangle,fill=white]{5};
			\draw (4,4) node[left, rectangle,fill=white]{4};
			\draw (4,1) node[right, rectangle,fill=white]{1};
			\draw[thick,*->-*](1,3)--(1,5);
			\draw[thick,*->-*](2,5)--(2,2);
			\draw[thick,*->-*](3,2) node[left]{2}--(3,4);
			\draw[thick,*->-*](4,4)--(4,1);
			\foreach \x in {1,2,...,4}
			\draw (-.1,\x-.5) node[yshift=.5cm]{$\ubar{\x}$};
		\end{tikzpicture} 
		}
				\caption{The edge diagram for $35241$ has one irreducible interval: $[1,5]$.}
	\end{subfigure} \hskip 5mm
	\begin{subfigure}{.3\textwidth}
	\centering
		\small{
		\begin{tikzpicture} [scale=.5]
			\draw (.5,1)--(6.5,1);		
			\draw (.5,4)--(6.5,4);
			\draw (.5,7)--(6.5,7);
			\draw (1,1) node[left, rectangle,fill=white]{1};
			\draw (2,3) node[left, rectangle,fill=white]{3};
			\draw (4,7) node[left, rectangle,fill=white]{7};
			\draw (5,5) node[left, rectangle,fill=white]{5};
			\draw (6,6) node[left, rectangle,fill=white]{6};
			\draw (6,4) node[rectangle, fill=white, right]{4};
			\draw[thick,*->-*](1,1)--(1,3);
			\draw[thick,*->-*](2,3)--(2,2);
			\draw[thick,*->-*](3,2) node[left]{2}--(3,7);
			\draw[thick,*->-*](4,7)--(4,5);
			\draw[thick,*->-*](5,5)--(5,6);
			\draw[thick,*->-*](6,6)--(6,4);
			\foreach \x in {1,2,...,6}
			\draw (-.1,\x-.5) node[yshift=.5cm]{$\ubar{\x}$};
		\end{tikzpicture}
		}
		\caption{The edge diagram for $1327564$ has two irreducible intervals: $[1,4]$ and $[4,7]$.} \label{sepinsep}
	\end{subfigure}
	\caption{Examples of irreducible intervals} \label{irredintex}
\end{figure}

Lemma \ref{irreducible}  allows us to consider signed permutations of the irreducible intervals of $\pi$, rather than unsigned permutations of the levels, since every $\sigma$ that is valid with respect to $\pi$ permutes and possibly flips its irreducible blocks. Recall that a signed permutation is a bijection $\mu: [-n] \cup [n] \rightarrow [-n] \cup [n]$ (where $[-n] \cup [n]= \{-n,\ldots,-1,1,\ldots,n \}$) such that $\mu(-i)=-\mu(i)$ for all $i \in [n]$.  We write the signed permutation $\mu$ as $\mu(1)\mu(2) \ldots \mu(n)$, leaving out the images of the negative numbers since they follow, and writing $\overline{\mu(i)}$ instead of $-\mu(i)$.  Note that 
the barring operation is an involution, i.e. $\overline{\overline{\mu(i)}}=\mu(i)$.  Let $|\mu(i)|$ denote the entry $\mu(i)$ without a bar. We denote the set of signed permutations of length $k$ by $B_k$, and the infinite set of all signed permutations by~$B$.

If $\pi \in S_n$ has $k$ irreducible intervals, we can apply a signed permutation $\mu \in B_k$ to the edge diagram of $\pi$ as follows. Each entry of $\mu$ describes where the corresponding irreducible interval of $\pi$ is moved, and a barred entry indicates that  the order of the levels inside the irreducible interval is reversed.  As an example, see Figure \ref{signedex}.  If the result of applying $\mu$ to an edge diagram is a well-defined edge diagram, we denote it by $\EE_{\mu.\pi}$.  If, additionally, this is a proper edge diagram, we say that $\mu$ is \emph{valid} with respect to~$\pi$ and denote the corresponding permutation by $\mu.\pi$.

 We need to generalize our notion of a valid flip to signed permutations. If $[x_i,x_j]$ is a valid interval in $\EE_\pi$, then we say that 
$\mu =1 \, 2 \, 3 \, \cdots \, (i-1) \, (\overline{j-1}) \, (\overline{j-2}) \, \cdots \,  \overline{i} \, j \, \cdots \, k\in B_k$ is a \emph{valid flip} with respect to $\pi$. 
 We say that $\mu$ transforms $\pi$ by a \emph{sequence of valid flips} if there exist $\mu_1, \mu_2, \ldots, \mu_l \in B_{k}$ with $\mu_l \mu_{l-1} \cdots \mu_1=\mu$ such that $\mu_i$ is a valid flip with respect to $\mu_{i-1} \mu_{i-2} \ldots \mu_1. \pi$ for all $i$.

\begin{figure}[htb] 
	\centering \footnotesize{
	\begin{tikzpicture} [scale=.55]
		\draw (.5,1)--(9.5,1);
		\draw (.5,3)--(9.5,3);		
		\draw (.5,6)--(9.5,6);
		\draw(.5,10)--(9.5,10);
		\draw (1,2) node[left, rectangle,fill=white]{2};
		\draw (2,1) node[left, rectangle,fill=white]{1};
		\draw (3,3) node[left, rectangle,fill=white]{3};
		\draw (4,5) node[left, rectangle,fill=white]{5};
		\draw (5,4) node[left, rectangle,fill=white]{4};
		\draw (6,6) node[left, rectangle,fill=white]{6};
		\draw (7,9) node[left, rectangle,fill=white]{9};
		\draw (8,7) node[left, rectangle,fill=white]{7};
		\draw (9,8) node[left, rectangle,fill=white]{8};
		\draw (9,10) node[rectangle, fill=white, right]{10};
		\draw[thick,*->-*](1,2)--(1,1);
		\draw[thick,*->-*](2,1)--(2,3);
		\draw[thick,*->-*](3,3)--(3,5);
		\draw[thick,*->-*](4,5)--(4,4);
		\draw[thick,*->-*](5,4)--(5,6);
		\draw[thick,*->-*](6,6)--(6,9);
		\draw[thick,*->-*](7,9)--(7,7);
		\draw[thick,*->-*](8,7)--(8,8);
		\draw[thick,*->-*](9,8)--(9,10);	
		\draw (-.2,2) node{$\mathbf{1}$};	
		\draw (-.2,4.5) node{$\mathbf{2}$};
		\draw (-.2,8) node{$\mathbf{3}$};
		\draw (12,5) node{$\xmapsto{\mu=\bar{3} 1 \bar{2}}$};				
	\end{tikzpicture}
	\begin{tikzpicture} [scale=.45]
		\draw (.5,1)--(9.5,1);
		\draw (.5,4)--(9.5,4);		
		\draw (.5,8)--(9.5,8);
		\draw(.5,10)--(9.5,10);
		\draw (1,1) node[left, rectangle,fill=white]{1};
		\draw (2,3) node[left, rectangle,fill=white]{3};
		\draw (3,2) node[left, rectangle,fill=white]{2};
		\draw (4,4) node[left, rectangle,fill=white]{4};
		\draw (5,6) node[left, rectangle,fill=white]{6};
		\draw (6,7) node[left, rectangle,fill=white]{7};
		\draw (7,5) node[left, rectangle,fill=white]{5};
		\draw (8,8) node[left, rectangle,fill=white]{8};
		\draw (9,10) node[left, rectangle,fill=white]{10};
		\draw (9,9) node[rectangle, fill=white, right]{9};
		\draw[thick,*->-*](1,1)--(1,3);
		\draw[thick,*->-*](2,3)--(2,2);
		\draw[thick,*->-*](3,2)--(3,4);
		\draw[thick,*->-*](4,4)--(4,6);
		\draw[thick,*->-*](5,6)--(5,7);
		\draw[thick,*->-*](6,7)--(6,5);
		\draw[thick,*->-*](7,5)--(7,8);
		\draw[thick,*->-*](8,8)--(8,10);
		\draw[thick,*->-*](9,10)--(9,9);	
		\draw (-.2,2.5) node{$\mathbf{2}$};	
		\draw (-.2,6) node{$\mathbf{\overline{3}}$};
		\draw (-.2,9) node{$\mathbf{\overline{1}}$};					
	\end{tikzpicture} }
	\caption{$2\,1\,3\,5\,4\,6\,9\,7\,8\,10$ is mapped to $1\,3\,2\,4\,7\,5\,6\,8\,10\,9$ by applying $\mu=\bar{3} 1 \bar{2}$} \label{signedex}
\end{figure}
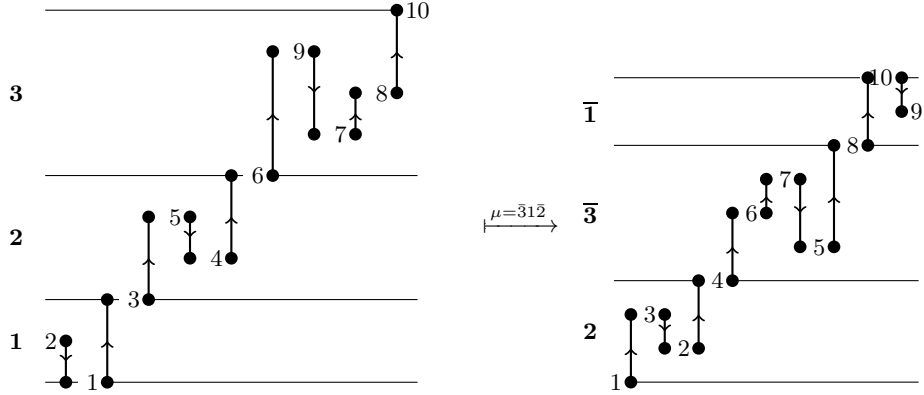

Given $\mu \in B_k$, its \emph{reverse-complement} is the signed permutation with $\mu^{RC}(i)=\overline{\mu(k-i+1)}$ for $1 \leq i \leq k$.  For $\alpha_1,\alpha_2,\ldots,\alpha_k \in B$, we define the \emph{inflation} of $\mu$ by $\alpha_1,\alpha_2,\ldots,\alpha_k$ to be the signed permutation obtained by replacing each $\mu(i)$ with a block that has the pattern $\alpha_i$ if $\mu(i)$ is positive, and $\alpha_i^{RC}$ if $\mu(i)$ is negative.  The relative ordering of the blocks is determined by the relative unsigned ordering on $\mu$ (meaning we consider simply absolute values).  We denote this inflation by $\mu[\alpha_1,\alpha_2,\ldots,\alpha_k]$.
 For example, $3 \bar{1} 2 [\bar{2} 1,\bar{3} 1 \bar{2}, 1]=\bar{6} 5 \, 2 \bar{1} 3 \, 4$.

We can now restate our main theorem using the irreducible partition.

\begin{thm}\label{main theorem irred}
Let $\pi, \tau \in S_n$ where $\pi$ has $k$ irreducible intervals.  Then $\pi \sim \tau$ if and only if there exists some $\mu \in B_k$ such that $\EE_{\mu.\pi}=\EE_{\tau}$ and $\mu$ transforms $\pi$ by a sequence of valid flips.
\end{thm}

\section{Cohesive Intervals and Partitions} \label{Cohesive Intervals and Partitions}

In this section, we introduce the notion of a cohesive partition, a generalization of the irreducible partition.  We are going to induct on the number of blocks of a partition of the edge diagram of $\pi$, so the purpose of this generalization is to have a method of coarsening the irreducible partition.

\begin{definition}
Let $\pi \in S_n$ with $k$ irreducible intervals, and let $[a,b]$ be an interval in the edge diagram of $\pi$.  We say that $[a,b]$ is \emph{cohesive} in $\pi$ if all of the following are true:
\begin{enumerate}[(1)]
\item The interval $[a,b]$ is a union of irreducible intervals;
\item for every $e \in \EE_{\pi}$, we have $e \subseteq [a,b]$, $[a,b] \subseteq e$, or $[a,b] \cap e = \emptyset$; and
\item for every $\mu \in B_{k}$ that is valid with respect to $\pi$, $\mu. [a,b]$ is an interval.
\end{enumerate}
\end{definition}

Condition (3) is difficult to check, since one would in principle have to verify the property for all signed permutations of length $k$.  Notice that every edge in $\EE_\pi$ satisfies condition (3), but most will fail (1) or (2).  All irreducible intervals satisfy conditions (1) and (3), but not necessarily condition (2), as is the case for interval $[1,4]$ in Figure~\ref{irredintex}(c) . Some examples of cohesive intervals are given in Figure~\ref{cohesive}. Since cohesive intervals are unions of irreducible intervals, we will sometimes use the borders $x_i$ of the irreducible partition to describe them. For convenience, we write 
$\{a_0,a_1,\dots,a_l\}_<$ to denote the partition into intervals $[a_i,a_{i+1}]$ for $0\le i<l$.

\begin{figure}[htb]
\centering \footnotesize{
	\begin{tikzpicture} [scale=.5]
		\draw (.5,1)--(9.5,1);		
		\draw (.5,2)--(9.5,2);
		\draw (.5,5)--(9.5,5);
		\draw (.5,6)--(9.5,6);
		\draw(.5,9)--(9.5,9);
		\draw (1,1) node[left, rectangle,fill=white]{1};
		\draw (2,9) node[left, rectangle,fill=white]{9};
		\draw (3,7) node[left, rectangle,fill=white]{7};
		\draw (4,8) node[left, rectangle,fill=white]{8};
		\draw (5,6) node[left, rectangle,fill=white]{6};
		\draw (6,2) node[left, rectangle,fill=white]{2};
		\draw (7,4) node[left, rectangle,fill=white]{4};
		\draw (8,3) node[left, rectangle,fill=white]{3};
		\draw (8,5) node[rectangle, fill=white, right]{5};
		\draw[thick,*->-*](1,1)--(1,9);
		\draw[thick,*->-*](2,9)--(2,7);
		\draw[thick,*->-*](3,7)--(3,8);
		\draw[thick,*->-*](4,8)--(4,6);
		\draw[thick,*->-*](5,6)--(5,2);
		\draw[thick,*->-*](6,2)--(6,4);
		\draw[thick,*->-*](7,4)--(7,3);
		\draw[thick,*->-*](8,3)--(8,5);						
	\end{tikzpicture} }
	\caption{In the permutation $1\,9\,7\,8\,6\,2\,4\,3\,5$, all the irreducible intervals are cohesive: $ [1,2]$, $[2,5]$, $[5,6]$, $[6,9]$.  Additionally, the intervals $[2,6]$ and $[1,9]$  are cohesive.}
\label{cohesive}
\end{figure}
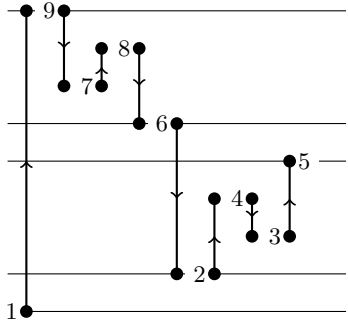


\begin{definition}
A partition $\mathcal{P}=\{x_s=a_0,a_1,a_2, \ldots, a_l=x_t \}_<$ of a cohesive  interval $[x_s,x_t]$ is called a \emph{cohesive partition} if $[a_i,a_{i+1}]$ is irreducible or cohesive for every $1\le i<l$.
\end{definition}

In a cohesive interval, the finest cohesive partition will be the irreducible partition $\mathcal{I}_{\pi}$ restricted to $[x_s,x_t]$.  Cohesive partitions provide us a more general setting to prove our main theorem: we will consider signed permutations of the blocks in $\mathcal{P}$ that result in a valid edge diagram, and show that all such permutations are sequences of valid flips.  Then, since the interval $[1,n]$ is trivially cohesive and has cohesive partition into irreducible intervals, Theorem~\ref{main theorem irred} will follow as a particular case.

In the rest of this section, we assume that $\pi \in S_n$ has $k$ irreducible intervals, and we use
$\mathcal{P}=\{a_0,a_1,a_2, \ldots, a_l\}_<$ to denote a cohesive partition of the cohesive interval $[x_s,x_t]$.
In the following definition, $1$ stands for the identity permutation of length one and $\mathds{1}_t$ stands for the identity permutation of length $t$. We write $1^l$ to denote a sequence of $l$ ones. 

\begin{definition}
We say that $\mu \in B_l$ is valid with respect to $[x_s,x_t]_{\PP}$ if $\mathds{1}_{k-t+s+1}[1^s,\mu',1^{k-t}] \in B_k$  is valid with respect to $\pi$, where $\mu'=\mu[\mathds{1}_{m_0},\mathds{1}_{m_1},\ldots,\mathds{1}_{m_{l-1}}] \in B_{t-s}$ and $m_i$ is the number of irreducible intervals contained in $[a_i,a_{i+1}]$.
We denote the image of $[x_s,x_t]$ under the transformation of $\mu$ by $\mu.[x_s,x_t]_{\PP}$.
\end{definition}

Cohesive partitions play an important role in our proof. In some cases, $\mathcal{P}$ contains a \emph{proper cohesive interval}, which is a cohesive interval $[a_i,a_j]$ where $j\neq i+1$ and $(i,j)\neq(0,l)$.  In these cases, we are able to create two new cohesive partitions, $\mathcal{P}'=\{a_0,a_1,\ldots,a_i,a_j,\ldots,a_l \}_<$ and $\widetilde{\mathcal{P}}=\{a_i,a_{i+1},\ldots,a_j \}_<$, and use them to decompose the action of a valid $\mu$, as described in the following lemma.

\begin{lem}
 \label{decompose lemma}
Let $\mu \in B_l$ be valid with respect to $[x_s,x_t]_{\mathcal{P}}$. If $[a_i,a_j]$ is a proper cohesive interval, there exist $\alpha \in B_{l-j+i+1}$ and $\beta \in B_{j-i}$ such that $\mu = \alpha[1^{i},\beta,1^{l-j}]$.  Additionally, $\beta$ is valid with respect to $[a_i,a_j]_{\widetilde{\PP}}$.
\end{lem}

The idea of the proof, which is omitted in this extended abstract, is that there are only two obvious choices for our $\alpha,\beta$ pair. In one of them, $\alpha(i)$ is positive, and in the other it is negative. One can show that one of these choices makes $\beta$ a valid permutation.  Since $\mu$ is valid with respect to $[x_s,x_t]_{\PP}$, this amounts to examining the behavior of the path at the vertices $a_i$ and $a_j$ after applying $\beta$.  Showing that there is at most one indegree and outdegree at both $a_i$ and $a_j$ is enough to show that the result is a path.

We now have all the tools needed to handle the main theorem. In fact we prove the following more general statement about cohesive partitions.

\begin{prop} \label{valid flips}
Every $\mu \in B_l$ that is valid with respect to $[x_s,x_t]_{\mathcal{P}}$ transforms $[x_s,x_t]_{\mathcal{P}}$ by a sequence of valid flips.
\end{prop}

The idea of the proof, again omitted, is that if $\mathcal{P}$ contains a proper cohesive interval $[a_i,a_j]$, we can then use induction to find a sequence of valid flips that first permutes the intervals inside of $[a_i,a_j]$ and then the intervals outside of $[a_i,a_j]$.  The difficulty  comes when there is no proper cohesive interval. We deal with this case separately and directly in the proof.

Theorem~\ref{main theorem irred}, and thus its equivalent restatements Theorems~\ref{main theorem} and~\ref{Main Thm}, follow now as a consequence of Proposition \ref{valid flips}.

\section{Future Work} \label{Future Work}


We have defined a natural equivalence $\sim$ on permutations and characterized the equivalence classes using valid flips. We conjecture that our equivalence classes describe precisely when two permutations are obtained with the same probability in a random walk regardless of the probability distribution on the steps. We are currently working on proving this statement.

\begin{conj}
For $\pi, \tau \in S_n$, $\PPP(p(X_1,X_2,\ldots,X_{n-1})=\pi)=\PPP(p(X_1,X_2,\ldots,X_{n-1})=\tau)$ for every probability distribution on the i.i.d. random variables $X_1,X_2,\ldots,X_{n-1}$ if and only if $\pi \sim \tau$.
\end{conj}

\section*{Acknowledgements}
The authors thank Adeline Pierrot for useful ideas during the early stages of this work. The first author was partially supported by grant \#280575 from the Simons Foundation and by grant H98230-14-1-0125 from the NSA.

\bibliographystyle{plain}
\bibliography{Thesis_References}

\end{document}